\documentclass[a4paper, 11pt, twoside, notitlepage]{amsart}

\usepackage{amsmath,amscd}
\usepackage{amssymb}
\usepackage{amsthm}
\usepackage{comment}
\usepackage{graphicx}
\usepackage{epstopdf} 
\usepackage{mathrsfs}
\usepackage[ocgcolorlinks, linkcolor=blue]{hyperref}

\usepackage[ocgcolorlinks,linkcolor=blue]{hyperref}

\newcommand{\LC}{\left(}
\newcommand{\RC}{\right)}

\theoremstyle{plain}
\newtheorem{thm}{Theorem}[section]
\newtheorem{prop}{Proposition}[section]

\newtheorem{rmk}[prop]{Remark}

\numberwithin{equation}{section}
\newcommand {\R} {\mathbb{R}} 
 \newcommand {\N} {\mathbb{N}}
 
\newcommand {\p} {\partial}
\newcommand{\eps}{\epsilon}

\newcommand {\supp} {\text{supp}}

\newcommand{\norm}[1]{\lVert #1 \rVert}

\newcommand{\wt}{\widetilde}

\newcommand{\kommentar}[1]{}

\title[Inverse Problems for Fractional PDEs]{Inverse problems for fractional equations with a minimal number of measurements}

\author[Y.-H. Lin]{Yi-Hsuan Lin}
\address{Department of Applied Mathematics, National Yang Ming Chiao Tung University, Hsinchu, Taiwan}
\email{yihsuanlin3@gmail.com}

\author[H. Liu]{Hongyu Liu}
\address{Department of Mathematics, City University of Hong Kong, Kowloon, Hong Kong SAR, China}
\email{hongyu.liuip@gmail.com, hongyliu@cityu.edu.hk}

\begin{document}
	
	\maketitle
	
	\begin{abstract}
	
	In this paper, we study several inverse problems associated with a fractional differential equation of the following form:
	\[
	(-\Delta)^s u(x)+\sum_{k=0}^N a^{(k)}(x) [u(x)]^k=0,\ \ 0<s<1,\ N\in\mathbb{N}\cup\{0\}\cup\{\infty\},
	\]
	which is given in a bounded domain $\Omega\subset\mathbb{R}^n$, $n\geq 1$. For any finite $N$, we show that $a^{(k)}(x)$, $k=0,1,\ldots, N$, can be uniquely determined by $N+1$ different pairs of Cauchy data in $\Omega_e:=\mathbb{R}^n\backslash\overline{\Omega}$. If $N=\infty$, the uniqueness result is established by using infinitely many pairs of Cauchy data. The results are highly intriguing in that it generally does not hold true in the local case, namely $s=1$, even for the simplest case when $N=0$, a fortiori $N\geq 1$. The nonlocality plays a key role in establishing the uniqueness result. We also establish several other unique determination results by making use of a minimal number of measurements. Moreover, in the process we derive a novel comparison principle for nonlinear fractional differential equations as a significant byproduct.  	
		
		\medskip
		
		\noindent{\bf Keywords.}  Fractional differential equations, semilinear, inverse problems, minimal number of measurements, comparison principle.  
		
	%	\noindent{\bf Mathematics Subject Classification (2010)}: 
		
	\end{abstract}

    \tableofcontents

	\section{Introduction}
	
	\subsection{Mathematical setup and statement of main results}
	
	Let $\Omega\subset\mathbb{R}^n$, $n\geq 1$, be bounded domain with a $C^{1,1}$ boundary $\partial\Omega$ and $\Omega_e:=\mathbb{R}^n\backslash\overline{\Omega}$. Let $s\in (0, 1)$. Consider the following fractional differential equation:
	\begin{align}\label{main equation fractional}
		\begin{cases}
			(-\Delta)^s u +\mathbf{a}(x,u) =F &\text{ in }\ \Omega, \\
			u=f &\text{ in }\ \Omega_e,
		\end{cases}
	\end{align}
	where $F\in L^\infty(\Omega)$ and $f\in C^{2,s}(\Omega_e)$. The function $\mathbf{a}(x, u)$, $(x, u)\in\Omega\times\mathbb{R}$, takes the following form:
	\begin{align}\label{condition a polynomial}
	\mathbf{a}(x,u)=\sum_{k=1}^N a^{(k)}(x)u^k,\ \ a^{(k)}(x)\in C^{s}(\overline\Omega), \ N\in\mathbb{N}\cup\{\infty\}. 
         \end{align}
	Here, two remarks are in order. First, when $N=\infty$, we need to require the convergence of the infinite series in \eqref{condition a polynomial} in the $C^s(\overline{\Omega})$-topology for $u$ sufficiently small. It shall be made explicitly in what follows. Second, one can conveniently write $a^{(0)}(x):=-F(x)$ and set
	\begin{equation}\label{eq:nl1}
	\widetilde{\mathbf{a}}(x, u):=\sum_{k=0}^N a^{(k)}(x)u^k=\mathbf{a}(x, u)-F(x). 
	\end{equation}
	In doing so, the fractional equation in \eqref{main equation fractional} can be more compactly written as $(-\Delta)^s u+\widetilde{\mathbf{a}}(x, u)=0$. Nevertheless, since the regularity assumptions on $a^{(k)}(x)$, $k\geq 1$ and $a^{(0)}(x)=F(x)$ are different and moreover, in the physical setup, $F$ is usually referred to as a source term and $a^{{(k)}}(x)$'s are medium parameters, we stick to the formulation in \eqref{main equation fractional}-\eqref{condition a polynomial} in our subsequent study. 
	
In \eqref{main equation fractional}, $(-\Delta)^s$ is the fractional Laplacian operator given by (cf. \cite{di2012hitchhiks})
	\begin{align}\label{fractional Laplacian}
		(-\Delta)^s u(x):=c_{n,s}\mathrm{P.V.}\int_{\R^n} \frac{u(x)-u(y)}{|x-y|^{n+2s}}\, dy,
	\end{align}
	where 
	\begin{align*}%\label{c_ns}
		c_{n,s}=\frac{\Gamma(\frac{n}{2}+s)}{|\Gamma(-s)|}\frac{4^{s}}{\pi^{n/2}}
	\end{align*}
	is a dimensional constant depending on $n\in \N$, $s\in (0,1)$, and P.V. stands for the principal value. Let $W_1$ and $W_2$ be two arbitrary nonempty open subsets in $\Omega_e$. It is always assumed that $\mathrm{supp}(f)\subset W_1$, and moreover $f|_{W_1}\in C^{2,s}(\overline{W_1})$. In Section~\ref{Sec 2}, we shall establish the well-posedness of the forward problem \eqref{main equation fractional}, especially the unique existence of a solution $u\in C^s(\mathbb{R}^n)$ under general conditions. Assuming such a well-posedness, we introduce the following exterior nonlocal partial Cauchy data set:  
    \begin{equation}\label{Cauchy data}
    \begin{split}
    	\mathcal{C}_{\mathbf{a},F}^{W_1,W_2} (f):=& \LC \left. u \right|_{W_1},  \left. (-\Delta)^s u \right|_{W_2} \RC\\
	=& \LC \left. f\right|_{W_1},  \left. (-\Delta)^s u \right|_{W_2} \RC\\
	\subset & \LC C^{2,s}_0(W_1) \cap \wt H^s(W_1)\RC \times C^s(W_2) , 
    \end{split}
    \end{equation}
where $u\in C^s(\mathbb{R}^n)$ is the unique solution to \eqref{main equation fractional}. Here $C^{2,s}_0$ and $\wt H^s$ are the H\"older space and fractional Sobolev space, which will be introduced in Section \ref{Sec 2}.

In what follows, we shall simply call $\mathcal{C}_{\mathbf{a},F}^{W_1,W_2} (f)$ the pair of Cauchy data associated with an external input $f$. For a particular case with $f\equiv 0$, 
\begin{align*}
		\mathcal{C}_{\mathbf{a},F}^0:=\mathcal{C}_{\mathbf{a},F}^{W_1, W_2}(0)=\LC 0|_{W_1},  \left. (-\Delta)^s u  \right|_{W_2} \RC
	\end{align*}
	is purely generated by the source term $F$. In the physical setup, $\mathcal{C}_{a, F}^0$ is referred to as the {\it passive measurement}. In comparison, for a nontrivial input $f$, $\mathcal{C}_{\mathbf{a},F}^{W_1,W_2} (f)$ is referred to as an {\it active measurement}. In this article, we are concerned with the following inverse problem:
	
	\medskip
	
	\noindent \textbf{Inverse Problem.} Can one determine $\mathbf{a}(x,u)$ of the form \eqref{condition a polynomial}, namely $a^{(k)}(x)$, $k\geq 1$, and $F(x)$ by using the Cauchy data $\mathcal{C}_{\mathbf{a},F}^{W_1,W_2}(f)$ associated with different inputs $f$?
	
	\medskip
	
The first main result that we establish for the above inverse problem is the unique determination of $F(x)$ and $a^{(k)}(x)$, $k=1, 2, \ldots, N$, by using $N+1$ pairs of Cauchy data $\mathcal{C}_{\mathbf{a},F}^{W_1,W_2} (f_k)$ associated with $f_0, f_1,\ldots, f_N$, where $f_k\equiv\hspace*{-3mm}\backslash\, f_l$, $0\leq k, l\leq N$ and $k\neq l$. Here, by $f_k\equiv\hspace*{-4mm}\backslash\, f_l$ we mean that there exists at least one point $x\in W_1$ such that $f_k(x)\neq f_l(x)$, and hence by the continuity, $f_k$ and $f_l$ take different values in an open subset of $W_1$. This represents a minimal number of measurements in the sense that the number of unknown functions equals to the number of the measured Cauchy data pairs, being both $N+1$. The unique determination result can be stated as follows. 

\begin{thm}\label{Main Thm 1}
		Let $\Omega \subset\R^n$ be a bounded domain with a $C^{1,1}$ boundary $\p \Omega$, for $n\geq 1$, and $s\in (0,1)$.  Let $W_1,W_2 \subset \Omega_e$ be two arbitrary nonempty open subsets, and consider 
		\begin{align}\label{equation in thm fractional}
			\begin{cases}
			(-\Delta)^s u_j + \mathbf{a}_j(x,u_j)= F_j & \text{ in }\ \Omega,\\
				u_j=f &\text{ in }\ \Omega_e,
			\end{cases}
		\end{align}
		where $\mathbf{a}_j=\mathbf{a}_j(x,u)$ is of the form \eqref{condition a polynomial} with coefficient functions $a_j^{(k)}$, for $k=1,2,\ldots, N$, $j=1,2$ and a finite $N\in\mathbb{N}$. Assuming the well-posedness of \eqref{equation in thm fractional},  if
		\begin{align}\label{Cauchy fractional 1}
			\mathcal{C}^{W_1,W_2}_{\mathbf{a}_1,F_1}(f_k)=\mathcal{C}^{W_1,W_2}_{\mathbf{a}_2,F_2}(f_k),\ k=0,1,\ldots, N,
		\end{align}
where $f_k\equiv\hspace*{-3.3mm}\backslash\, f_l$, $0\leq k, l\leq N$ and $k\neq l$, then one has
	    \[
	    a_1^{(k)}(x)=a_2^{(k)}(x) \text{ in }\Omega,\ k=1,2,\ldots, N, \quad \text{ and }\quad F_1=F_2 \text{ in }\Omega.
	    \]
	\end{thm}
	
	\begin{rmk}
	In Theorem~\ref{Main Thm 1}, the well-posedness of the forward problem \eqref{equation in thm fractional} is assumed. This shall be proved in Section~\ref{Sec 2}. It turns out that we need to impose more restrictive assumptions including $a^{(1)}\geq 0$, $x\in\Omega$, and  additional smallness conditions on $f$ and $F$, especially when $N\geq 2$. On the other hand, if $N\leq 1$, i.e. the forward problem \eqref{equation in thm fractional} is linear, the regularity assumptions on $\partial\Omega, F, a^{(1)}$ and $f$ can be much relaxed. This shall also be remarked in Section~\ref{Sec 2}. It is emphasized that those conditions are only needed for the forward problem and will not affect the inverse problem study as long as the well-posedness of the forward problem is guaranteed. This should be more clear in our subsequent analysis. 
	\end{rmk}
	
Next, we consider passing $N$ to $\infty$ in Theorem~\ref{Main Thm 1}. It is natural to conjecture that one can establish the unique determination result by countably many exterior measurements, namely $\mathcal{C}_{\mathbf{a},F}^{W_1,W_2}(f_j), j\in\mathbb{N}$. However, establishing such a uniqueness result is fraught with significant difficulties. This is not unusual when one passes from finite unknowns to infinite unknowns in the study of inverse problems due to their ill-posed nature; see e.g. \cite{blaasten2016corners,liu2019mosco,KLW2021calder,harrach2020monotonicity}. Instead, we prove a unique determination result by using uncountably many measurements in the case $N=\infty$. This corroborates the sharpness of our result in Theorem~\ref{Main Thm 1}. We shall make more discussion about this point in the next subsection. Before that, we make more rigorous about the convergence of the infinite series in \eqref{condition a polynomial} when $N=\infty$. In such a case, we always assume that the infinite series is absolutely convergent in $C^s(\overline{\Omega})$-topology when $|u|<\varepsilon$ for a sufficiently small $\varepsilon\in\mathbb{R}_+$. This in fact means that $\mathbf{a}(\cdot, u)$ is real analytic around $u=0$ with values in $C^s(\overline\Omega)$. Moreover, throughout this paper, we shall always assume in such a case that 
\begin{align}\label{condition a positive}
	a^{(1)}(x)=\partial_u\mathbf{a}(x, 0)\geq 0   \text{ for }x\in \Omega,
\end{align}
again for the sake of guaranteeing the well-posedness of the forward problem. 

\begin{thm}\label{Main Thm 2}
Let $\Omega \subset\R^n$ be a bounded domain with a $C^{1,1}$ boundary $\p \Omega$, for $n\geq 1$, and $s\in (0,1)$. Let $X_\varepsilon$ and $Y_\delta$ be two function spaces that will be introduced by \eqref{local well-posedness source} and \eqref{local well-posedness exterior}, respectively. Suppose that $F_j\in X_\varepsilon$ and $f\in Y_\delta$.  Consider the following semilinear equation:
		\begin{align}\label{equation in thm4}
			\begin{cases}
			(-\Delta)^s u_j + \mathbf{a}_j(x,u_j)= F_j & \text{ in }\ \Omega,\\
				u_j=f &\text{ in }\ \Omega_e,
			\end{cases}
		\end{align}
		where $\mathbf{a}_j=\mathbf{a}_j(x,u)$ is of the form \eqref{condition a polynomial} with $N=\infty$ and satisfying the assumptions described above. 

 Let $W_1,W_2 \subset \Omega_e$ be two arbitrary nonempty open subsets, and $\mathcal{C}_{\mathbf{a}_j,F_j}^{W_1,W_2}$ be the associated Cauchy data of \eqref{equation in thm4}, for $j=1,2$. If it holds that 
	\begin{align}\label{Cauchy fractional 2}
		\mathcal{C}_{\mathbf{a}_1,F_1}^{W_1,W_2}(f)=\mathcal{C}_{\mathbf{a}_2,F_2}^{W_1,W_2}(f), \text{ for any }f\in Y_\delta,
	\end{align}
	then one has
	\[
	\mathbf{a}_1=\mathbf{a}_2\ \text{ in }\ \Omega\times \R, \quad \text{ and } \quad F_1=F_2\ \text{ in }\ \Omega.
	\]
\end{thm}

It is noted that in \eqref{Cauchy fractional 2}, we need to make use of the Cauchy data for all the inputs from the function space $Y_\delta$. In fact, the inputs depend on an asymptotic parameter $\delta$, which means that we need to make use of uncountably many pairs of Cauchy data in Theorem~\ref{Main Thm 2}. It is noted that if $F_1=F_2=0$, the unique determination result has been proved in \cite{LL2020inverse}. We shall follow a similar successive linearization technique in deriving the uniqueness result. As mentioned earlier, we present such a unique determination result mainly to corroborate the novelty and significance of the result in Theorem~\ref{Main Thm 1} as well as its argument which make full use of the finite dimensionality and nonlocality of the underlying problem. 

Finally, it would be interesting to consider the same inverse problem in Theorem~\ref{Main Thm 1}, but with a reduced number of unknown functions. To illustrate, let us consider a specific case by assuming that $N=2$ in Theorem~\ref{Main Thm 1}, but $a_1^{(1)}=a_2^{(1)}=0$. That is, for the inverse problem of recovering $\mathbf{a}$ and $F$ by knowledge of $\mathcal{C}_{\mathbf{a}, F}^{W_1, W_2}$, it is sufficient to recover $a^{(2)}$ and $F$ since $a^{(1)}=0$ is a-priori known. Clearly, by Theorem~\ref{Main Thm 1}, one can obtain the unique determination by using 3 different pairs of Cauchy data. However, since the number of unknown functions is reduced to be 2, it is natural to ask whether one can achieve the uniqueness result by 2 measurements, namely 2 Cauchy pairs. We shall give an affirmative answer to such a specific case.

\begin{thm}\label{Thm: qudratic}
Let $\Omega \subset\R^n$ be a bounded domain with a $C^{1,1}$ boundary $\p \Omega$, for $n\geq 1$, and $s\in (0,1)$. Let $X_\varepsilon$ and $Y_\delta$ be the same function spaces in Theorem~\ref{Main Thm 2}. Let $a_j^{(2)}\in C^s(\overline{\Omega})$ be nonnegative functions and $F_j\in X_\varepsilon$ be nonpositive sources, for $j=1,2$. Consider the following semilinear equation:  
 	\begin{align}\label{equation in thm 3}
 		\begin{cases}
 			(-\Delta)^s u_j+a_j^{(2)}u_j^2= F_j & \text{ in }\ \Omega,\\
 			u_j=f &\text{ in }\ \Omega_e,
 		\end{cases}
 	\end{align}
 	for $j=1,2$. Let $W_1,W_2\subset \Omega_e$ be two arbitrarily nonempty open subsets, and $\mathcal{C}_{a_j^{(2)},F_j}^{W_1,W_2}$ be the Cauchy data of \eqref{equation in thm 3}, for $j=1,2$. Let $f\in Y_\delta \setminus \{0\}$ be a nonpositive function with $\supp\, (f)\subset W_1$. If it holds that   
 	\begin{align}\label{DN fractional 3}
 		\mathcal{C}^0_{a_1^{(2)},F_1}=\mathcal{C}^0_{a_{2}^{(2)},F_2} \quad \text{ and }\quad \mathcal{C}_{a_1^{(2)},F_1}^{W_1,W_2}(f)=\mathcal{C}_{a_2^{(2)},F_2}^{W_1,W_2}(f),
 	\end{align}
then one has 
 	\[
 	a_1^{(1)}=a_2^{(2)} \quad  \text{ and } \quad  F_1=F_2\ \text{ in }\ \Omega.
 	\]
 \end{thm}
 
 As can be seen from Theorem~\ref{Thm: qudratic}, the number of the unknown functions and the number of the Cauchy data pairs are equal to each other, being both 2. Hence, if the number of unknown functions is reduced in Theorem~\ref{Main Thm 1}, one in principle should be able to reduce the number of measurements to be minimal. Nevertheless, as can also be seen from Theorem~\ref{Thm: qudratic}, one needs to impose certain a-priori conditions on the unknowns. In fact, even in such a specific case, we need to develop a different argument from that for the proof of Theorem~\ref{Main Thm 1}. This partly evidences the significant challenge by deriving unique determination results using a minimal number of measurements. In principle, one can extend the result in Theorem~\ref{Thm: qudratic} by considering a general unknown function $\mathbf{a}(x, u)$ of the form \eqref{condition a polynomial} with some of its coefficient functions a-priori known. The proof can follow a similar argument to that for Theorem~\ref{Thm: qudratic} in Section~\ref{Sec 4}. However, presenting such a general result involves rather lengthy and tedious explanations, especially on the a-priori conditions imposed on the unknown coefficient functions. Hence, instead of being swamped with nontechnical details, we choose to present Theorem~\ref{Thm: qudratic} only to verify the point that if the number of unknown functions is reduced, one may still be able to reduce the number of measurements to be minimal, under the situation that certain a-priori conditions are imposed on the unknown functions.

\subsection{Discussion and connection to existing studies}

Inverse problems for fractional PDEs (partial differential equations), namely differential equations involving fractional PDOs (partial differential operators), have received considerable studies in the literature recently. The fractional type Calder\'on problem was first investigated in \cite{ghosh2016calder}. It turns out that one can solve several challenging inverse problems that still remain unsolved in local cases, namely the fractional PDOs are replaced by the corresponding non-fractional counterparts, by taking advantage of the nonlocality of the fractional PDEs. We refer to \cite{bhattacharyya2021inverse,GLX,CL2019determining,CLL2017simultaneously,cekic2020calderon,feizmohammadi2021fractional,harrach2017nonlocal-monotonicity,harrach2020monotonicity,GRSU18,GU2021calder,ghosh2021non,lin2020monotonicity,LL2020inverse,KLW2021calder,RS17,ruland2018exponential} and the references cited therein for the existing developments in the literature.

Most of the existing studies are concerned with inverse problems for linear fractional PDEs by using infinitely (actually uncountably) many measurements. In \cite{CLL2017simultaneously}, fractional Schiffer's problem was proved by using a single measurement. In \cite{GRSU18}, uniqueness was derived by a single measurement in determining a potential term associated with a fractional Schr\"odinger equation, which is actually a special case of Theorem~\ref{Main Thm 1} with $N=1$ and $F\equiv 0$. There are also related studies on inverse problems associated with semilinear fractional PDEs \cite{lai2019global,lin2020monotonicity,LL2020inverse}, all making use of uncountably many measurements. 

On the other hand, inverse problems for nonlinear PDEs enjoys a similar hotness as the fractional ones in the literature recently; see \cite{FO19,feizmohammadi2021inverse,harrach2022simultaneous,KU2019partial,KU2019remark,LLLS2019nonlinear,LLLS2019partial,LLST2020inverse,LL2020inverse,LLL2021determining,LLLZ2021simultaneous} and the references cited therein. A major tool is the so-called successive linearization or higher order linearization technique, which attaches some asymptotic parameters to the measurement data and in nature requires uncountably many measurements. It is noted that the aforementioned studies on inverse problems for fractional semilinear PDEs made use of successive linearization technique as well. 

Inverse problems with minimal/optimal measurement data have been a central topic in the theory of inverse problems, which possesses both theoretical and practical significance. On the other hand, it is always highly challenging to achieve unique determination results by using minimal/optimal measurement data. We mention the so-called Schiffer's problem \cite{blaasten2021scattering,liu2006uniqueness,liu2008uniqueness,liu2020local} and partial-data Caler\'on's problem \cite{kenig2007calderon,kenig2014calderon,imanuvilov2010calderon,bai2021effective} as typical and classical examples in the local setting. Hence, our results in Theorems~\ref{Main Thm 1} and \ref{Thm: qudratic} are highly intriguing. First, they make use of a minimal number of measurements as discussed before Theorem~\ref{Main Thm 1}. Second, in deriving those inversion results, we do not make use of any linearization technique, even dealing with nonlinear PDEs. The linearization is only employed for showing the well-posedness of the forward problem. Instead, we explore the nonlocality and make full use of it. Our study as well as the method developed in the present article open up a new direction of research for fractional inverse problems with minimal/optimal measurement data.

Finally, we would like to mention another interesting connection of our study. As mentioned earlier, $F$ in \eqref{main equation fractional} signifies a source term, whereas $a^{(k)}$'s in \eqref{condition a polynomial} usually signify certain medium parameters in the physical scenario. Inverse problems of simultaneously recovering an unknown source inside a body as well as the medium parameters of the body by using both passive and active measurements have also received considerable attentions in the literature. They usually possess strong backgrounds of applications; see \cite{liu2015determining} in photo- and thermo-acoustic tomography, \cite{deng2019identifying,deng2020identifying} in geomagnetic anomaly detection, \cite{KLU2018,LLL2021determining} in cosmological exploration, \cite{li2019determining,li2021determining} in quantum scattering as well as \cite{LLLZ2021simultaneous} for more related ones. In particular, we would like to mention that in \cite{CL2019determining} the authors consider the simultaneous recovery of an unknown source and its surrounding medium parameter associated with a linear fraction PDE by uncountably many measurements.

The rest of the paper is organized as follows. In Section \ref{Sec 2}, we review several function spaces, which were be used in our study of inverse  problems. Meanwhile, we also study the well-posedness for both linear and nonlinear fractional elliptic equations. We prove Theorems \ref{Main Thm 1}, when the function $\mathbf{a}(x,u)$ is linear with respect to the $u$-variable in Section \ref{Sec 3}. In Section \ref{Sec 4}, we prove Theorem \ref{Main Thm 1} without using any linearization method, and show that Theorem \ref{Main Thm 2} holds by using the higher order linearization.  Finally, we demonstrate a particular case of an inverse source problem with reduced unknowns, and we can determine both coefficients and source by using two measurements in Section \ref{Sec 5}.

 \section{Preliminaries}\label{Sec 2}
 
 In this section, we prove the well-posedness of the Dirichlet problem \eqref{main equation fractional}, under different settings of the function $\mathbf{a}=\mathbf{a}(x,u)$. Before doing so, let us review H\"older spaces and fractional Sobolev spaces.

 \subsection{Function spaces}
 
 Let us review definitions of several function spaces, which will be used in the proofs of our results. 
 Let $U\subset\R^n$ be an open set, $k\in \N \cup \{0\}$ and $0<\alpha <1$. The H\"older space $C^{k,\alpha}(U)$ is defined by
 \[
 C^{k,\alpha}(U):=\left\{f:U\to \R:\ \norm{f}_{C^{k,\alpha}(U)}<\infty \right\}.
 \]
 The norm $\norm{\cdot}_{C^{k,\alpha}(U)}$ is given by 
 \begin{align*}
 	\norm{f}_{C^{k,\alpha}(U)}:=\sum_{|\beta|\leq k}\norm{\p ^\beta f}_{L^\infty(U)} +\sum_{|\beta|=k}[\p ^\beta f]_{C^\alpha(U)},
 \end{align*}
 where 
 \[
 [\p ^\beta f]_{C^\alpha(U)}:=\sup_{\begin{subarray}{c}
 		x\neq y, \\
 		x,y\in U
 \end{subarray}}\frac{|\p ^\beta f(x)-\p^\beta f(y)|}{|x-y|^\alpha}
 \]
 denotes the seminorm of $C^{0,\alpha}(U)$, $\beta=(\beta_1,\ldots,\beta_n)$ is a multi-index with $\beta_i \in \N \cup \{0\}$ and $|\beta|=\beta_1 +\ldots +\beta_n$. 
 Furthermore, we also denote the space
 \[
 C_0^{k,\alpha}(U):=\text{closure of }C^\infty_c(U) \text{ in }C^{k,\alpha}(U).
 \]
Set $C^\alpha(U) \equiv C^{0,\alpha}(U)$ when $k=0$.

 We turn to recall the definition of fractional Sobolev spaces. Given $\gamma\in \R$, the $L^2$-based fractional Sobolev space $H^{\gamma}(\mathbb{R}^{n}):=W^{\gamma,2}(\mathbb{R}^{n})$ is defined via 
 \begin{equation}\notag
 	\|u\|_{H^{\gamma}(\mathbb{R}^{n})}=\LC \|u\|_{L^{2}(\mathbb{R}^{n})}^{2}+\|(-\Delta)^{\gamma/2}u\|_{L^{2}(\mathbb{R}^{n})}^{2}\RC^{1/2}.\label{eq:H^s norm}
 \end{equation}
 Furthermore, via the Parseval identity, the semi-norm $\|(-\Delta)^{\beta/2}u\|_{L^{2}(\mathbb{R}^{n})}$
 can be rewritten as 
 \[
 \|(-\Delta)^{\gamma/2}u\|_{L^{2}(\mathbb{R}^{n})}=\left((-\Delta)^{\gamma}u,u\right)_{\mathbb{R}^{n}}^{1/2},
 \]
 where $(-\Delta)^\gamma $ is the fractional Laplacian, for $\gamma\in \R$.
 
 Given an open set  $U\subset \R^n $ and $\gamma\in\mathbb{R}$,
 then we denote the following Sobolev spaces, 
 \begin{align*}
 	H^{\gamma}(U) & :=\left\{u|_{U}:\, u\in H^{\gamma}(\mathbb{R}^{n})\right\},\\
 	\widetilde{H}^{\gamma}(U) & :=\text{closure of \ensuremath{C_{c}^{\infty}(U)} in \ensuremath{H^{\gamma}(\mathbb{R}^{n})}},\\
 	H_{0}^{\gamma}(U) & :=\text{closure of \ensuremath{C_{c}^{\infty}(U)} in \ensuremath{H^{\gamma}(U)}},
 \end{align*}
 and 
 \[
 H_{\overline{U}}^{\gamma}:=\left\{u\in H^{\gamma}(\mathbb{R}^{n}):\,\mathrm{supp}(u)\subset\overline{U}\right\}.
 \]
 In addition, the fractional Sobolev space $H^{\gamma}(U)$ is complete with respect to the norm
 \[
 \|u\|_{H^{\gamma}(U)}:=\inf\left\{ \|v\|_{H^{\gamma}(\mathbb{R}^{n})}:\,v\in H^{\gamma}(\mathbb{R}^{n})\mbox{ and }v|_{U}=u\right\} .
 \]
 Moreover, when $U$ is a Lipschitz domain, the dual space of $H^\gamma(U)$ can be written as 
 \begin{align*}
 	\LC H^{\gamma}_{\overline{U}}(\R^n)\RC ^\ast = H^{-\gamma}(U), \quad \text{ and }\quad \LC H^{\gamma}(U)\RC^\ast=H^{-\gamma}_{\overline U}(\R^n).
 \end{align*}
 We also denote $\LC \wt H^{\gamma}(U)\RC^\ast$ to be the dual space of $\wt H^{\gamma}(U)$.  For those readers who are interested to know more properties for fractional Sobolev spaces, see \cite{di2012hitchhiks,mclean2000strongly} for more details.

 \subsection{Forward problems}
 
Let us begin with the linear case, i.e., 
$$
\mathbf{a}(x,u)=a^{(1)}(x)u
$$
in \eqref{main equation fractional} (as $N=1$), where $a^{(1)}\in L^\infty(\Omega)$, and the (global) well-posedness was shown in \cite[Lemma 2.3]{ghosh2016calder} by using the Lax-Milgram theorem. Consider the following eigenvalue condition 
 \begin{align}\label{eq:eigenvalue condition}
 	0 \text{ is not a Dirichlet eigenvalue of }(-\Delta)^s +a^{(1)} \text{ in }\Omega,
 \end{align}
 then the well-posedness for linear fractional Schr\"odinger equation holds.

 \begin{prop}[Well-posedness for the fractional Schr\"odinger equation]\label{Prop: Well-posedness}
 	Let $\Omega \subset \R^n$ be a bounded open set, for $n\geq 1$, and $a^{(1)}\in L^\infty(\Omega)$. Let $f\in H^s(\R^n)$ and $F\in \LC \wt H^s(\Omega) \RC ^\ast$. Suppose that \eqref{eq:eigenvalue condition} holds, then there exists a unique solution $u\in H^s(\R^n)$ of
 	\begin{align*}
 		\begin{cases}
 			\LC (-\Delta)^s +a^{(1)} \RC u=F &\text{ in }\Omega,\\
 			u=f &\text{ in }\Omega_e,
 		\end{cases}
 	\end{align*}
 	such that 
 	\[
 	\norm{u}_{H^s(\R^n)}\leq C \LC \norm{F}_{\LC \wt H^s(\Omega)\RC^\ast} +\norm{f}_{H^s(\R^n)} \RC, 
 	\]
 	for some constant $C>0$ independent of $u$, $f$, and $F$.
 \end{prop}
 
 For the semilinear counterparts, we can prove the (local) well-posedness of the fractional elliptic equation \eqref{main equation fractional}, where $\mathbf{a}(x,u)$ is of the form \eqref{condition a polynomial}, whenever the Dirichlet data and sources are sufficiently small in appropriate function spaces.

 Let us introduce the following function spaces
 \begin{align}\label{local well-posedness source}
 	X_\varepsilon :=\left\{ F\in L^\infty(\Omega): \, \norm{F}_{L^\infty(\Omega)}<\varepsilon \right\}, 
 \end{align}
 \begin{align}\label{local well-posedness exterior}
 	\begin{split}
 		Y_\delta:= \left\{ f\in Y : \, \norm{f}_{Y} <\delta \right\},
 	\end{split}
 \end{align}  
 where 
 \[
 Y:=C^{2,s}_0(\Omega_e)\cap \wt H^s(\Omega_e)
 \]
 and  $\norm{\cdot}_Y$ is given by 
 $$
 \norm{f}_Y:=\norm{f}_{C^{2,s}_0(\Omega_e)}+ \norm{f}_{\wt H^s(\Omega_e)}.
 $$
 Given $s\in (0,1)$, let $\mathbf{a}(x,u)$ satisfy \eqref{condition a polynomial}, $F\in X_\varepsilon$ and $f\in Y_\delta$, we can prove that the local well-posedness of 
 \begin{align}\label{main equation_nonlinear quadratic}
 	\begin{cases}
 		(-\Delta)^s u +\mathbf{a}(x,u) = F & \text{ in }\Omega,\\
 		u=f &\text{ in }\Omega_e,
 	\end{cases}
 \end{align}
 for sufficiently small $\varepsilon, \delta>0$.
 Let us state the known well-posedness result for linear fractional Schr\"odinger equation without the proof.

 \begin{thm}[Local well-posedness for the fractional semilinear elliptic equation]\label{Thm local well-posedness}
 	Let $\Omega \subset\R^n$ be a bounded domain with $C^{1,1}$ boundary $\p \Omega$. Let $\mathbf{a}(x,u)$ satisfy  \eqref{condition a polynomial}, $F \in X_\epsilon$ and $f\in Y_\delta$, for sufficiently small $\varepsilon$ and $\delta$.	
 	Consider the fractional semilinear elliptic equation 
 	\begin{align}\label{eq well-posedness}
 		\begin{cases}
 			(-\Delta)^s u + \mathbf{a}(x,u)=F &\text{ in }\Omega,\\
 			u=f&\text{ in }\Omega_e.
 		\end{cases}
 	\end{align}
 	Then there exists a unique solution $u\in C^{s}(\R^n)$ of \eqref{eq well-posedness}. In addition, the solution $u$ satisfies 
 	\begin{align*}
 		\norm{u}_{C^s(\R^n)} \leq C\LC \norm{F}_{L^\infty(\Omega)} + \norm{f}_{Y} \RC,
 	\end{align*}
 	for some constant $C>0$ independent of $u$, $f$ and $F$.
 \end{thm}

 \begin{proof}
 	The proof is similar to the proof of \cite[Proposition 2.1]{LL2022uniqueness}.
 	Let 
 	\[
 	V:=\left\{ u\in H^s(\R^n): \,  (-\Delta)^s u\in L^\infty \Omega , \,  u|_{\Omega} \in C^s(\Omega),\, u|_{\Omega_e} \in C^{2,s}(\Omega_e) \right\},
 	\]
 	and 
 	$$
 	\norm{\phi}_{V}:=\norm{\phi}_{H^s(\R^n)}+\norm{(-\Delta)^s \phi}_{L^\infty(\Omega)}+\norm{\phi}_{C^s(\Omega)}+ \norm{\phi}_{C^{2,s}(\Omega_e)},
 	$$
 	then it is not hard to see that $\LC  V, \norm{\cdot}_V \RC$ is a Banach space.
 	Consider the map
 	\begin{align*}
 		\Phi: V\times L^\infty(\Omega) \times Y \to& L^\infty(\Omega) \times V,\\
 		(u,F,f)\mapsto  &  \LC (-\Delta)^s u +\mathbf{a}(x,u)-F, u|_{\Omega_e}-f  \RC,
 	\end{align*}
 	then $\Phi(0,0,0)=(0,0)$, where we utilized the condition \eqref{condition a polynomial} such that $\mathbf{a}(x,0)=0$ in $\Omega$. Similar to the proofs of (local) well-posedness in \cite{KU2019partial,LLLS2019nonlinear,LL2020inverse,LL2022uniqueness}, we can show that the first linearization of $\Phi=\Phi(u,F,f)$ at $(0,0,0)$ with respect to the $u$-variable is
 	\begin{align*}
 		\p_u\Phi(0,0,0):V\to & L^\infty(\Omega)\times V , \\
 		v \mapsto & \LC (-\Delta)^s v +\p_y\mathbf{a}(x,0)v, v|_{\Omega_e} \RC,
 	\end{align*}
 	which is isomorphism via the (global) well-posedness for linear fractional Schr\"odinger equation (c.f. Proposition \ref{Prop: Well-posedness}).

 	Notice that $\LC Y, \norm{\cdot}_Y \RC$ is also a Banach space. Finally, by utilizing the implicit function theorem for Banach spaces (see \cite[Chapter 10]{renardy2006introduction} for example), there is an open neighborhood $X_\varepsilon \times Y_\delta $ of $(0,0)$ in $L^\infty(\Omega)\times Y$ and a unique $C^\infty$-smooth function $h: X_\varepsilon \times Y_\delta\to V$ such that 
 	\[
 	\Phi(h(F,f),F,f)=(0,0),
 	\]
 	where $(F,f)\in \LC X_\varepsilon , Y_\delta  \RC$, for sufficiently small $\varepsilon, \delta>0$.
 	Since $h$ is smooth, and $h(0,0)=0$, the solution $u=h(F,f)$ satisfies 
 	\[
 	\norm{u}_V\leq C\LC \norm{F}_{L^\infty(\Omega)}+\norm{f}_Y \RC\leq C\LC \varepsilon+\delta\RC,
 	\]
    for some constant $C>0$ independent of $u,F$ and $f$.
 	Finally, indeed, it is not hard to check the map $\Phi$ is a $C^\infty$ Fr\'echet differentiable map as we want (see \cite{LLLS2019nonlinear} for more detailed arguments). This prove the assertion.
 \end{proof}

 In the end of this preliminary section, let us recall the global strong unique continuation principle and the Runge approximation for fractional elliptic equations (cf. \cite{ghosh2016calder,GRSU18,FF2014unique}).

 \begin{prop}[Global strong unique continuation property]\label{Prop:(Uniqueness-theorem)} 
 	Given $s\in(0,1)$, $n\in \N$, let $u\in H^{s}(\mathbb{R}^{n})$ be a function satisfying $u=(-\Delta)^s u=0$
 	on some positive measurable set $\mathcal{O}$ of $\mathbb{R}^{n}$. Then $u\equiv0$ in $\mathbb{R}^{n}$. 
 \end{prop}
 
 \begin{prop}[Runge approximation property]
 	\label{Prop:(Runge-approximation-property)} 
 	Let $\Omega\subset\mathbb{R}^{n}$ be a bounded open set and $D\subseteq\mathbb{R}^{n}$
 	be an arbitrarily open set such that $\Omega \Subset D$, for $n\in \N$.
 	Let $a^{(1)}\in L^{\infty}(\Omega)$ satisfies \eqref{eq:eigenvalue condition},
 	then for any $f\in L^{2}(\Omega)$, for any $\epsilon>0$, we can
 	find a function $u_{\epsilon}\in H^{s}(\mathbb{R}^{n})$ which solves
 	\[
 	\LC (-\Delta)^s+a^{(1)}\RC  u_{\epsilon}=0\mbox{ in }\Omega\quad \mbox{ and }\quad \mathrm{supp}\, (u_{\epsilon})\subseteq\overline{D}
 	\]
 	and 
 	\[
 	\|u_{\epsilon}-f\|_{L^{2}(\Omega)}<\epsilon.
 	\]
 \end{prop}

 \section{Inverse problems for linear fractional elliptic equations}\label{Sec 3}
 
 In this section, let us prove Theorem \ref{Main Thm 1} for the linear case, that is, consider the function 
 \begin{align}\label{linear a in sec 3}
 	 \mathbf{a}(x,u)=a^{(1)}(x)u,
 \end{align}
 to be a linear function with respect to $u\in \R$.
 In particular, if $a^{(1)}=0$ in $\Omega$, we can simply use the passive measurement to determine the source.
 
 \begin{thm}[Unique determination by the passive measurement]\label{Thm: Linear passive}
 	Let $\Omega \subset\R^n$ be a bounded domain with Lipschitz boundary $\p \Omega$, for $n\geq 1$, and $s\in (0,1)$.  Given an open subset $W_2\subset \Omega_e$, let $\mathcal{C}_{0,F_j}^0$ be the passive measurements 
 	\begin{align}\label{equ passive one}
 		\begin{cases}
 			(-\Delta)^s u_j^{(0)} =F_j &\text{ in }\Omega, \\
 			u_j^{(0)}=0 &\text{ in }\Omega_e,
 		\end{cases}
 	\end{align}
 	for $j=1,2$. If  
 	\begin{align}\label{Cauchy: Linear passive}
 		\mathcal{C}_{0,F_1}^0=\mathcal{C}_{0,F_2}^0 \text{ in }W_2, 
 	\end{align}
 	then $F_1=F_2$ in $\Omega$.
 \end{thm}

 \begin{proof}
 	The proof is based on the global strong unique continuation property (Proposition \ref{Prop:(Uniqueness-theorem)}). Let $u_j^{(0)}\in H^s(\R^n)$ be the solution of \eqref{equ passive one}, and the condition \eqref{Cauchy: Linear passive} implies that $(-\Delta)^s \LC u_1^{(0)}-u_2^{(0)} \RC=0 $ in $W_2$. Combining with $u_1^{(0)}-u_2^{(0)}=0$ in $W_2$, and applying Proposition \ref{Prop:(Uniqueness-theorem)}, one must have $u_1^{(0)}=u_2^{(0)}$ in $\R^n$, which guarantees that 
 	\[
 	F_1 =(-\Delta)^s u_1^{(0)}=(-\Delta)^s u_2^{(0)}=F_2 \text{ in }\Omega
 	\]
 	as desired. This proves the assertion.
 \end{proof}

 \begin{rmk}\label{rmk:counterexample}
 	It is worth mentioning that:
 	\begin{itemize}
 		\item[(a)] 	It is known that without further assumptions for the source $F$, it is not possible to determine the source uniquely for local differential equations. A simple example  can be considered by the Poisson equation 
 		\begin{align}\label{equation local counter}
 			\begin{cases}
 				-\Delta u_j =F_j &\text{ in }\Omega,\\
 				u_j=f& \text{ on }\p \Omega,
 			\end{cases}
 		\end{align}
 		for $j=1,2$. In fact, to find the obstruction for the unique determination problem, let $\phi \in C_c^{2}(\Omega)$ be an arbitrary function, then one has $\phi =\p_\nu \phi =0$ on $\p \Omega$. Let $\LC  u_j|_{\p \Omega},\p_\nu u_j|_{\p \Omega} \RC$ be the Cauchy data of \eqref{equation local counter}, even if 
 		$$
 		\LC  u_1|_{\p \Omega},\, \p_\nu u_1|_{\p \Omega} \RC = \LC  u_2|_{\p \Omega},\, \p_\nu u_2|_{\p \Omega} \RC,
 		$$ 
 		but we can always write $F_2=F_1-\Delta \phi$, and $\Delta\phi$ can be arbitrary. Therefore, the unique determination is not possible for the unknown sources  in general.
 		
 		\item[(b)] For the fractional case, we do not have an analogous counterexample similar to the local case. The reason is that we cannot find any nontrivial function $\phi$ such that the Dirichlet data and Neumann data are zero in the measured domain in the fractional setup. If there exists a such function $\phi\in C^2_c(\Omega)$, with the same Cauchy data, then we must have $\phi=(-\Delta)^s\phi=0$ in some open subset of $\Omega_e$. This implies that $\phi \equiv 0$ in $\R^n$ (see Proposition \ref{Prop:(Uniqueness-theorem)}). Hence, it is possible to determine the unknown source uniquely via exterior measurements for fractional inverse source problems.
 	\end{itemize}
 \end{rmk}

 We next study the linear fractional Schr\"odinger equation with source, and consider the function \eqref{linear a in sec 3} (\eqref{condition a polynomial} for $N=1$). As a matter of fact, for the linear case, we do not need to assume the coefficient $a^{(1)}\in C^s(\overline{\Omega})$, but we assume $a^{(1)}\in C^0(\overline{\Omega})$, which is a continuous function.

 \begin{thm}[Unique determination by two measurements]\label{Thm: Linear two meas}
 	Let $\Omega \subset\R^n$ be a bounded domain with Lipschitz boundary $\p \Omega$, for $n\geq 1$, $s \in (0,1)$, and $a^{(1)}_j\in C^0(\overline{\Omega})$. Consider the following fractional Schr\"odinger equation
 	\begin{align}\label{equation in linear two meas}
 		\begin{cases}
 			\LC (-\Delta)^s+a^{(1)}_j\RC u_j= F_j & \text{ in }\Omega,\\
 			u_j=f &\text{ in } \Omega_e,
 		\end{cases}
 	\end{align}
 	for $j=1,2$. Given arbitrarily open subsets $W_1,W_2\subset\Omega_e$, and an exterior data $f\in \wt H^s(W_1)\setminus \{0\}$, let $\mathcal{C}_{a^{(1)}_j,F_j}^{W_1,W_2}(f)$ and $\mathcal{C}^0_{a^{(1)}_j,F_j}$ be the Cauchy data and passive measurement of \eqref{equation in linear two meas}, respectively. Then 
 	\begin{align}\label{DN fractional 2}
 		\mathcal{C}_{a^{(1)}_1,F_1}^0=\mathcal{C}_{a^{(1)}_2,F_2}^0 \quad  \text{ and }	 \quad 	\mathcal{C}_{a^{(1)}_1,F_1}^{W_1,W_2}(f)=\mathcal{C}_{a^{(1)}_2,F_2}^{W_1,W_2}(f)
 	\end{align}
 	implies that 
 	\[
 	a^{(1)}_1=a^{(1)}_2 \quad \text{ and }\quad F_1=F_2 \text{ in }\Omega.
 	\]
 \end{thm}

 \begin{proof}
 	Given an exterior data $f\in \wt H^s(W_1)\setminus \{0\}$, let $u_j$ and $u_j^{(0)}$ be the solutions of \eqref{equation in linear two meas} and
 	\begin{align*}
 		\begin{cases}
 			\LC (-\Delta)^s +a^{(1)}_j\RC u_j^{(0)} =F_j &\text{ in }\Omega, \\
 			u_j^{(0)}=0 &\text{ in }\Omega_e,
 		\end{cases}
 	\end{align*}
 	respectively, for $j=1,2$.  In addition, via the condition \eqref{DN fractional 2}, one has 
 	\begin{align}\label{same exterior Neumann}
 		(-\Delta)^s\LC u_1-u_2\RC =(-\Delta)^s \LC u_1^{(0)}-u_2^{(0)}\RC=0 \text{ in }W_2.
 	\end{align}

 	By subtracting the above two equations, we obtain that 
 	\begin{align}
 		\begin{cases}
 			\LC(-\Delta)^s +a^{(1)}_j\RC v_j=0 &\text{ in }\Omega,\\
 			v_j=f &\text{ in }\Omega_e,
 		\end{cases}
 	\end{align}
 	where $v_j:=u_j-u_j^{(0)}$ in $\R^n$, for $j=1,2$. By the identity \eqref{same exterior Neumann}, we obtain $(-\Delta)^s\LC v_1-v_2 \RC=0$ in $W_2$. On the other hand, we have $v_1-v_2=0$ in $\Omega_e$, since $a^{(1)}_j\in C^0(\overline{\Omega})$ for $j=1,2$, by applying \cite[Theorem 1]{GRSU18}, one can determine 
 	$$a^{(1)}_1=a^{(1)}_2 \text{ in } \Omega
 	$$ 
 	by using single measurement. Finally, similar to the proof of Theorem \ref{Main Thm 1}, it is not hard to see that 
 	\[
 	F_1 = \LC (-\Delta)^s +a^{(1)}_1 \RC u_1^{(0)}=\LC (-\Delta)^s +a^{(1)}_2 \RC u_2^{(0)} = F_2 \text{ in }\Omega,
 	\]
 	where we utilized $u_1^{(0)}=u_2^{(0)}$ in $\R^n$ (by Proposition \ref{Prop:(Uniqueness-theorem)} again). This completes the proof.
 \end{proof}
 
 \begin{rmk}
 	Note that Theorems \ref{Main Thm 1} and \ref{Main Thm 2} hold when the fractional Laplacian $(-\Delta)^s$ can be replace by more general nonlocal variable elliptic operator 
 	$$
 	\mathcal{L}^s=(-\nabla \cdot (\sigma \nabla))^s,\quad \text{ for }0<s<1,$$
 	where $\sigma=\LC\sigma_{ij}\RC_{1\leq i,j\leq n}$ is a $C^\infty$-smooth, positive definite, symmetric matrix-valued function. The key ingredient in the proof of Theorems \ref{Main Thm 1} and \ref{Main Thm 2} is based on the strong unique continuation property, and such property also holds for $\mathcal{L}^s$ (see \cite[Theorem 1.2]{GLX}). In addition, some inverse problems for this nonlocal variable coefficients operator $\mathcal{L}^s$ have been studied by \cite{CLL2017simultaneously,CL2019determining,GLX}.
 \end{rmk}

 \section{Inverse problems for semilinear fractional elliptic equations}\label{Sec 4}
 
 We prove Theorems \ref{Main Thm 1} and \ref{Main Thm 2} when $\mathbf{a}(x,u)$ are nonlinear with respect to $u\in \R$, which is of the form \eqref{condition a polynomial}.

 \subsection{Minimal number of measurements}

 With the local well-posedness (Theorem \ref{Thm local well-posedness}) at hand, we first prove Theorem \ref{Main Thm 1} with $(N+1)$ measurements. Before showing Theorem \ref{Main Thm 1}, let us study a special case as $N=2$, i.e.,
 \begin{align*}
 	\mathbf{a}(x,u)=a^{(1)}u+a^{(2)}u^2,
 \end{align*}
  in order to demonstrate the idea of the proof of Theorem \ref{Main Thm 1}.

 \begin{prop}[Three measurements]\label{Prop three measurements}
 	Let $\Omega \subset\R^n$ be a bounded domain with $C^{1,1}$ boundary $\p \Omega$, for $n\geq 1$, and $s\in (0,1)$.  Let $W_1,W_2 \subset \Omega_e$ be arbitrarily nonempty open subsets, and consider 
 	\begin{align}\label{equation 3 meas quadratic 0} 
 		\begin{cases}
 			(-\Delta)^s u_j + a_j^{(1)}u_j + a_j^{(2)}u_j^2= F_j & \text{ in }\Omega,\\
 			u_j=f &\text{ in } \Omega_e,
 		\end{cases}
 	\end{align}
 	where $a_j^{(\ell)}=a_j^{(\ell)}(x)\in C^s(\overline{\Omega})$, for $j,\ell=1,2$. Assuming that $F_j\in X_\varepsilon$, for sufficiently small $\varepsilon>0$, if 
 	\begin{align}\label{Cauchy equal 3 meas quadratic}
 		\mathcal{C}^{W_1,W_2}_{a_1,F_1}(f_\ell)=\mathcal{C}^{W_1,W_2}_{a_2,F_2}(f_\ell ), 
 	\end{align}
 	for $3$ different Dirichlet data $f_\ell\in Y_\delta$ with sufficiently small $\delta>0$, $\ell=0,1,2$,  then 
 	\[
 	a_1^{(1)}=a_2^{(1)}, \quad a_1^{(2)}=a_2^{(2)} \quad \text{ and }\quad F_1=F_2 \text{ in }\Omega.
 	\]
 \end{prop}

 \begin{proof}
 	Via Theorem \ref{Thm local well-posedness}, it is known that \eqref{equation 3 meas quadratic 0} is local well-posed, whenever $\varepsilon,\delta>0$ are sufficiently small. Let us choose $f_0=0$, and different $f_1,f_2\in Y_\delta$ to be not identically zero such that $\supp\, (f_1), \supp\, (f_2)\subset W_1$. Meanwhile, let $u_j^{(0)}$, $u_j^{(1)}$ and $u_j^{(2)}$ be the solutions of
 	\begin{align}\label{equation 3 meas quadratic}
 		\begin{cases}
 			(-\Delta)^s u_j^{(\ell)} + a_j^{(1)}u_j ^{(\ell)}+ a_j^{(2)}\LC u_j^{(\ell)}\RC^2= F_j & \text{ in }\Omega,\\
 			u_j^{(\ell)}=f_\ell &\text{ in } \Omega_e,
 		\end{cases}
 	\end{align}
 	for $\ell =0,1,2$.

 	By using \eqref{Cauchy equal 3 meas quadratic} and Proposition \ref{Prop:(Uniqueness-theorem)}, it is known that 
 	\begin{align}\label{unique sol in 3 meas qudratic}
 		u^{(\ell)}:=u_1^{(\ell)}=u_2^{(\ell)} \text{ in }\R^n, \text{ for }\ell =0,1,2.
 	\end{align}
 	On the other hand, by subtracting \eqref{equation 3 meas quadratic} between $\ell=0,1,2$, one has that 
 	\begin{align}\label{sub 1 in 3 meas}
 		\begin{cases}
 			(-\Delta)^s \LC u^{(1)}-u^{(0)} \RC +a_j^{(1)} \LC u^{(1)}-u^{(0)}\RC +  a_j^{(2)} \left[ \LC u^{(1)}\RC^2 -\LC u^{(0)} \RC^2  \right]=0  &\text{ in }\Omega,\\
 			u^{(1)}-u^{(0)}=f_1 &\text{ in }\Omega_e, \\
 		\end{cases}
 	\end{align}
 	and 
 	\begin{align}\label{sub 2 in 3 meas}
 		\begin{cases}
 		(-\Delta)^s \LC u^{(2)}-u^{(0)} \RC+	a_j^{(1)} \LC u^{(2)}-u^{(0)}\RC +  a_j^{(2)} \left[ \LC u^{(2)}\RC^2 -\LC u^{(0)} \RC^2  \right]=0 &\text{ in }\Omega,\\
 			u^{(2)}-u^{(0)}=f_2 &\text{ in }\Omega_e, \\
 		\end{cases}
 	\end{align}     
 	for $j=1,2$, where we used the uniqueness of solutions \eqref{unique sol in 3 meas qudratic}. 
 	Furthermore, subtracting \eqref{sub 1 in 3 meas} between $j=1,2$, we have 
 	\begin{align}\label{sub 3 in 3 meas}
 		\LC a_1^{(1)}-a_2^{(1)}\RC \LC u^{(1)}-u^{(0)}\RC + \LC a_1^{(2)}-a_2^{(2)}\RC \left[ \LC u^{(1)}\RC^2 -\LC u^{(0)} \RC^2  \right]=0 \text{ in }\Omega.
 	\end{align}
 	Similarly,  subtracting \eqref{sub 2 in 3 meas} between $j=1,2$ yields that 
 	\begin{align}\label{sub 4 in 3 meas}
 		\LC a_1^{(1)}-a_2^{(1)}\RC \LC u^{(2)}-u^{(0)}\RC +  \LC a_1^{(2)}-a_2^{(2)}\RC \left[ \LC u^{(2)}\RC^2 -\LC u^{(0)} \RC^2  \right]=0 \text{ in }\Omega.
 	\end{align}

 	Notice that both $u^{(1)}-u^{(0)}$ and $u^{(2)}-u^{(0)}$ are nonzero a.e. in $\R^n$ (by Proposition \ref{Prop:(Uniqueness-theorem)})\footnote{This can be seen via the $C^s$-H\"older continuity for solutions.}, so that \eqref{sub 1 in 3 meas} and \eqref{sub 2 in 3 meas} are equivalent to 
 	\begin{align}\label{sub 5 in 3 meas}
 		\begin{cases}
 			\LC a_1^{(1)}-a_2^{(1)}\RC + \LC a_1^{(2)}-a_2^{(2)}\RC\LC u^{(1)}+u^{(0)} \RC =0 &\text{ in }\Omega,\\
 			\LC a_1^{(1)}-a_2^{(1)}\RC+ \LC a_1^{(2)}-a_2^{(2)}\RC \LC u^{(2)}+u^{(0)}  \RC=0 &\text{ in }\Omega.
 		\end{cases}
 	\end{align}
 	To proceed, let us subtract \eqref{sub 5 in 3 meas}, then we have 
 	\begin{align}\label{point id in 3 meas}
 		\LC a_1^{(2)}-a_2^{(2)}\RC \LC u^{(1)}-u^{(2)}\RC =0 \text{ in }\Omega,
 	\end{align} 
 	and Proposition \ref{Prop:(Uniqueness-theorem)} yields that $u^{(1)}-u^{(2)}\neq 0$ a.e. in $\R^n$, since $f_1\neq f_2 $.
 	Hence, \eqref{point id in 3 meas} implies that $a_1^{(2)} - a_2^{(2)}=0$ a.e. in $\Omega$ such that 
 	\begin{align}\label{point id in 3 meas 1}
 		a_1^{(2)} - a_2^{(2)}=0 \text{ in }\Omega,
 	\end{align}
 	due to $a_j^{(2)}\in C^s(\overline{\Omega})$, for $j=1,2$. Let us plug \eqref{point id in 3 meas 1} into \eqref{sub 5 in 3 meas}, we can conclude that 
 	\begin{align}\label{point id in 3 meas 2}
 		a_1^{(1)}=a_2^{(1)} \text{ in }\Omega.
 	\end{align}
 	as desired. 
 	Finally, inserting \eqref{unique sol in 3 meas qudratic}, \eqref{point id in 3 meas 1} and \eqref{point id in 3 meas 2} into \eqref{equation 3 meas quadratic}, we obtain $F_1=F_2$ in $\Omega$. This completes the proof.
 \end{proof}

 We have shown that when there are three unknown factors, then we can use three different measurements to determine these unknowns. Let us next prove Theorem \ref{Main Thm 1}.
 
 \begin{proof}[Proof of Theorem \ref{Main Thm 1}]
 	We first express the function $a_j(x,y)$ of the form \eqref{condition a polynomial} as 
 	\begin{align*}
 		\mathbf{a}_j(x,y)=\sum_{k=1}^N a_j^{(k)}(x)y^k,
 	\end{align*}	
 	for $j=1,2$, then it remains to show that 
 	\begin{align}\label{claim equal coefficient in thm 1}
 		a_1^{(k)}=a_2^{(k)} \text{ in }\Omega, \text{ for }k=1,2,\ldots, N.
 	\end{align}
 	Let $f_0=0, f_1,\ldots, f_N\in Y_\delta$, which are mutually different, and consider $u_j^{(\ell)}$ to be the solutions of 
 	\begin{align}\label{equation in the proof of thm 1}
 		\begin{cases}
 			(-\Delta)^s u_j^{(\ell)} + \mathbf{a}_j(x,u_j^{(\ell)})= F_j & \text{ in }\Omega,\\
 			u_j^{(\ell)}=f_\ell &\text{ in } \Omega_e,
 		\end{cases}
 	\end{align}
 	for $\ell=0,1,\ldots, N$ and $j=1,2$.

 	Similar to the arguments in the proof of Proposition \ref{Prop three measurements}, with \eqref{Cauchy fractional 1} at hand, it is not hard to see that 
 	\begin{align}\label{unique sol in pf of thm1}
 		u^{(\ell)}:=u_1^{(\ell)}=u_2^{(\ell)} \text{ in }\R^n, \quad\text{ for }\ell=0,1,\ldots, N.
 	\end{align}
 	Moreover, via \eqref{equation in the proof of thm 1} and \eqref{unique sol in pf of thm1}, it is not hard to derive
 	\begin{align*}
 		\sum_{k=1}^N a_j^{(k)}\left[ \LC u^{(\ell)}\RC ^k-\LC u^{(0)} \RC^k \right]=0 \text{ in }\Omega,
 	\end{align*}
 	for $j=1,2$, so that 
 	\begin{align}\label{alg equ in pf of thm 1}
 		\sum_{k=1}^N \LC a_1^{(k)} -a_2^{(k)} \RC \left[ \LC u^{(\ell)}\RC^k-\LC  u^{(0)} \RC^k \right]=0 \text{ in }\Omega,
 	\end{align}
 	for all $\ell=0,1,\ldots, N$.
 	
 	Let us rewrite \eqref{alg equ in pf of thm 1} as 
 	\begin{align}\label{alg equ in pf of thm 1 2}
 		\mathbf{U}\mathbf{A}=0 \text{ in }\Omega, 
 	\end{align}
 	where $\mathbf{U}$ is an $N\times N$ matrix
 	\begin{equation}\label{matrix U}
 		\mathbf{U}:=\left( \begin{matrix}
 			u^{(1)} -u^{(0)} & \LC u^{(1)}\RC^2 -\LC u^{(0)} \RC^2 & \ldots & \LC u^{(1)} \RC^N -\LC u^{(0)} \RC^N\\
 			u^{(2)} -u^{(0)} & \LC u^{(2)}\RC^2 -\LC u^{(0)} \RC^2 & \ldots & \LC u^{(2)} \RC^N- \LC u^{(0)} \RC^N\\
 			\vdots&  \vdots &          \ddots        & \vdots  \\
 			u^{(N)} -u^{(0)} & \LC u^{(N)} \RC^2-\LC u^{(0)} \RC^2 & \ldots & \LC u^{(N)}\RC^N -\LC u^{(0)} \RC^N\end{matrix} \right) 
 	\end{equation}
 	and $\mathbf{A}$ is an $N$-column vector 
 	\begin{equation}\label{A-vector}
 		\mathbf{A}:=	\left( \begin{matrix}
 			a_1^{(1)}-a_2^{(1)}\\
 			a_1^{(2)}-a_2^{(2)}\\
 			\vdots          \\
 			a_1^{(N)}-a_2^{(N)}\\\end{matrix} \right).
 	\end{equation}
 	In order to show \eqref{claim equal coefficient in thm 1}, it suffices to show that the matrix $U$ in \eqref{matrix U} is non-singular a.e. in $\Omega$. 
 	
 	Via direct computations, we have 
 	\begin{align*}
 		\det \mathbf{U}=&\det \left( \begin{matrix}
 			u^{(1)} -u^{(0)} & \LC u^{(1)}\RC^2 -\LC u^{(0)} \RC^2 & \ldots & \LC u^{(1)} \RC^N -\LC u^{(0)} \RC^N\\
 			u^{(2)} -u^{(0)} & \LC u^{(2)}\RC^2 -\LC u^{(0)} \RC^2 & \ldots & \LC u^{(2)} \RC^N- \LC u^{(0)} \RC^N\\
 			\vdots&  \vdots &          \ddots        & \vdots  \\
 			u^{(N)} -u^{(0)} & \LC u^{(N)} \RC^2-\LC u^{(0)} \RC^2 & \ldots & \LC u^{(N)}\RC^N -\LC u^{(0)} \RC^N\end{matrix} \right) \\
 		=&\det \left( \begin{matrix}
 			1&    u^{(0)} & \LC u^{(0)} \RC^2 & \ldots & \LC u^{(0)} \RC^N\\	
 			0&	u^{(1)} -u^{(0)} & \LC u^{(1)}\RC^2 -\LC u^{(0)} \RC^2 & \ldots & \LC u^{(1)}\RC ^N -   \LC u^{(0)} \RC^N\\
 			0&	u^{(2)} -u^{(0)} & \LC u^{(2)} \RC^2 -\LC u^{(0)} \RC^2 & \ldots & \LC u^{(2)} \RC^N -\LC u^{(0)} \RC^N\\
 			\vdots&	\vdots&  \vdots &          \ddots        & \vdots  \\
 			0&	u^{(N)} -u^{(0)} & \LC u^{(N)} \RC^2 -\LC u^{(0)} \RC^2 & \ldots & \LC u^{(N)} \RC^N-\LC u^{(0)} \RC^N\end{matrix} \right) \\
 		=&\det \left( \begin{matrix}
 			1&    u^{(0)} & \LC u^{(0)} \RC^2 & \ldots & \LC u^{(0)} \RC^N\\	
 			1&	u^{(1)}  & \LC u^{(1)}\RC^2    & \ldots & \LC u^{(1)} \RC^N\\
 		   1&	u^{(2)} & \LC u^{(2)}  \RC^2 & \ldots & \LC u^{(2)}  \RC^N\\
 			\vdots&	\vdots&  \vdots &          \ddots        & \vdots  \\
 			1&	u^{(N)} & \LC u^{(N)}\RC^2 & \ldots & \LC u^{(N)} \RC^N\end{matrix} \right) ,
 	\end{align*}
 	which is the Vandermonde matrix in the last identity. Moreover, via the structure of the Vandermonde matrix, it is known that 
 	\begin{align*}
 		\det \mathbf{U} =\prod _{1\leq \ell<m\leq N} \LC u^{(m)}-u^{(\ell)}\RC .
 	\end{align*}
 	Now, since $u^{(m)}=f_m$ and $u^{(\ell)}=f_\ell$ in $\Omega_e$, with distinct exterior data, so we must have $u^{(m)}-u^{(\ell)}$ cannot be zero a.e. in $\Omega$, for all $m\neq \ell \in \{0,1,\ldots, N\}$, which shows $\det \mathbf{U}\neq 0$ a.e. in $\Omega$, so that $\mathbf{U}$ is non-singular a.e. in $\Omega$.

 	Therefore, one can conclude that the vector $\mathbf{A}$ in \eqref{A-vector} must be zero a.e. in $\Omega$. Finally, since each $a_j^{(k)}\in C^s(\overline{\Omega})$, for $j=1,2$, $k=1,2\ldots,N$, the claim \eqref{claim equal coefficient in thm 1} must hold. Finally, by using the equation \eqref{equation in the proof of thm 1}, we can summarize that $F_1=F_2$ in $\Omega$ as well, which proves the assertion.
 \end{proof}

 \begin{rmk}
 	Notice that in the both proof of Proposition \ref{Prop three measurements} and Theorem \ref{Main Thm 1}, we do not use any \emph{integral identities}, and one can determine unknowns by using finitely many measurements. These are major differences between local and nonlocal inverse problems.
 \end{rmk}

 \subsection{Successive linearization and proof of Theorem \ref{Main Thm 2}}
 Inspired by \cite{LLLS2019nonlinear,LLLS2019partial}, we want to adapt the higher order linearization method to show Theorem \ref{Main Thm 2}. let us consider the Dirichlet data
 \begin{align}\label{Dirichlet data for HOL}
 	f:=f(x;\eps):=\eps_1 g_1(x)+\eps_2g_2(x) \quad  \text{ in }\quad  \Omega_e,
 \end{align}
 where $g_k \in C^\infty_c(W_1)$, $\eps=(\eps_1, \eps_2)$, and $\eps_1,\eps_2$ are parameters such that every $|\eps_k|$ is  sufficiently small, for $k=1,2$. 
 
 Combining with \eqref{main equation fractional} and \eqref{Dirichlet data for HOL} at hand, we can apply the higher order linearization method to the semilinear elliptic equation
 \begin{align*}
 	\begin{cases}
 		(-\Delta)^s u +\mathbf{a}(x,u)=F & \text{ in }\Omega ,\\
 		u=\eps_1 g_1 +\eps_2g_2 & \text{ in } \Omega_e.
 	\end{cases}
 \end{align*}
 Since $\eps=(\eps_1,\eps_2)$, then $\eps=0$ is equivalent to $\eps_1 =\eps_2=0$.
 
 \begin{proof}[Proof of Theorem \ref{Main Thm 2}]
 	By using the special Dirichlet data \eqref{Dirichlet data for HOL}, let us denote $u_j$ to be the solutions of 
 	\begin{align}\label{semilinear in proof}
 		\begin{cases}
 			(-\Delta)^s u_j +\mathbf{a}_j(x,u_j)=F_j & \text{ in }\Omega ,\\
 			u_j=\eps_1g_1 +\eps_2g_2 & \text{ in } \Omega_e,
 		\end{cases}
 	\end{align}
 	for $j=1,2$.  Similar to the proof of Theorem \ref{Thm: Linear passive}, as $\eps=0$, let $u_j^{(0)}$ be the solution of 
 	\begin{align}\label{semilinear in proof with 0}
 		\begin{cases}
 			(-\Delta)^s u_j^{(0)} +\mathbf{a}_j(x,u_j^{(0)})=F_j & \text{ in }\Omega ,\\
 			u_j^{(0)}=0 & \text{ in } \Omega_e,
 		\end{cases}
 	\end{align}
 	for $j=1,2$. With the well-posedness at hand, we can apply the higher order linearization to show the unique determination result.

 	Let us first differentiate \eqref{semilinear in proof} with respect to $\eps_k$ around the function $u_j^{(0)}$, for $k=1,2$, then we have 
 	\begin{align}\label{equ: first order linear zero}
 		\begin{cases}
 			\LC   (-\Delta)^s +\p_y \mathbf{a}_j(x,u_j^{(0)}) \RC v_j^{(k)}=0 & \text{ in }\Omega,\\
 			v_j^{(k)} =f_k & \text{ on }\p \Omega,
 		\end{cases}
 	\end{align}
 	where 
 	\[
 	v_j^{(k)}=\left. \p_{\eps_k} \right|_{\eps=0}u_j,
 	\]
 	where $u_j^{(0)}$ is the solution of \eqref{semilinear in proof with 0}, for $j,k=1,2$. Due to the assumption of the coefficients $\mathbf{a}_j(x,y)$, one has that $\p_y \mathbf{a}_j(x,u_j^{(0)}(x))\in C^s(\overline{\Omega})$. Hence, via \cite[Theorem 1]{GRSU18} again, we are able to conclude that 
 	$$
 	A_1(x):=\p_y  \mathbf{a}_1(x,u_1^{(0)})=\p_y  \mathbf{a}_2(x,u_2^{(0)}), \text{ for }x\in \Omega.
 	$$
 	Furthermore, via the uniqueness for the (linear) fractional Schr\"odinger equation \eqref{equ: first order linear zero}, we can also conclude that
 	\begin{align*}%\label{uniqueness of the solutions to fractional Schrodinger equation}
 		v^{(k)}:=v_1^{(k)}=v_2^{(k)} \text{ in }\R^n, 
 	\end{align*}
 	for $k=1,2$. On the other hand, notice that $v^{(k)} \in C^s(\overline{\Omega})$ due to the global H\"older estimate for the fractional Schr\"odinger equation \eqref{equ: first order linear zero}.

 	We next derive the second linearization of \eqref{semilinear in proof} around the function $u_j^{(0)}$. Via straight forward computations, it is not hard to see that 
 	\begin{align}\label{equ: second order linear zero}
 		\begin{cases}
 			\LC (-\Delta)^s+A_1 \RC w^{(2)}_j + \p^2_y \mathbf{a}_j(x,u_j^{(0)}) \LC v^{(1)} \RC^2=0 & \text{ in }\Omega,\\
 			w^{(2)}_j=0 &\text{ on }\p \Omega,
 		\end{cases}
 	\end{align}
 	for $j=1,2$, where 
 	$$
 	w^{(2)}_j=\left.\p^2 _{\eps_1} \right|_{\eps=0}u_j,
 	$$
 	and $u_j$ is the solution of \eqref{semilinear in proof}.
 	Via the condition \eqref{DN fractional 3}, acting $\left.\p_{\eps_1}^2\right|_{\eps=0}$ on the identity $(-\Delta)^s u_1=(-\Delta)^s u _2 $ in $W_2$, one can get 
 	\[
 	(-\Delta)^s w^{(2)}_1 = (-\Delta)^s w^{(2)}_2 \text{ in }W_2.
 	\]
 	We next multiply \eqref{equ: second order linear zero} by $v^{(2)}$, and integrate over $\Omega$, then an integration by parts formula implies that 
 	\begin{align}\label{second integral id}
 		\int_{\Omega} \LC  \p^2_y \mathbf{a}_1(x,u_1^{(0)}) - \p^2_y \mathbf{a}_2(x,u_2^{(0)})\RC \LC v^{(1)} \RC^2v^{(2)}\, dx=0.
 	\end{align}
 	Now, by the Runge approximation (Proposition \ref{Prop:(Runge-approximation-property)}), for any $\varphi\in L^2(\Omega)$, there exists a sequence of functions $\left\{  f_2^{(\ell)} \right\}_{\ell=1}^\infty$ such that $v^{(2)}_\ell \to \varphi$ in $L^2(\Omega)$ as $\ell \to \infty$, where $v^2_{\ell}$ is the solution of 
 	\begin{align*}
 		\begin{cases}
 			\LC (-\Delta)^s +A_1\RC v^{(2)}_\ell =0  & \text{ in }\Omega, \\
 			v^{(2)}_\ell=f_2^{(\ell)} &\text{ in }\Omega_e,
 		\end{cases}
 	\end{align*}
 	for $\ell\in \N$. Replace $v^{(2)}$ by $v^{(2)}_\ell$ in \eqref{second integral id} and take $\ell\to\infty$, we then derive 
 	\[
 	\LC  \p^2_y \mathbf{a}_1(x,u_1^{(0)}) - \p^2_y \mathbf{a}_2(x,u_2^{(0)})\RC \LC v^{(1)} \RC^2=0 \text{ in }\Omega.
 	\]
 	By choosing the exterior data $f_1\geq 0 $ in $W_1$ but not identically zero, one can summarize that 
 	\[
 	\p^2_y \mathbf{a}_1(x,u_1^{(0)}) = \p^2_y \mathbf{a}_2(x,u_2^{(0)}) \text{ in }\Omega
 	\]
 	as desired. 
 	
 	Meanwhile, by using the mathematical induction hypothesis, it is not hard to show that 
 	\[
 	\p^k_y \mathbf{a}_1(x,u_1^{(0)})=\p^k_y \mathbf{a}_2(x,u_1^{(0)}) \text{ in }\Omega, 
 	\]
 	for all $k\geq 3$. The rest of the proof is similar to the one for \cite[Theorem 1.1]{LL2020inverse}.
 	Moreover, combining with the analytic assumption for the nonlinearity $\mathbf{a}_j(x,y)$ for $j=1,2$, one can conclude that 
 	\[
 	\mathbf{a}_1(x,y)=\mathbf{a}_2(x,y) \text{ in }\Omega \times \R.
 	\]
 	Finally, via the condition \eqref{DN fractional 3} again, we must have $u_1 \equiv u_2$ in $\R^n$ such that 
 	\[
 	F_1= (-\Delta)^s u_1 +\mathbf{a}_1(x,u_1)=(-\Delta)^s u_2 +\mathbf{a}_2(x,u_2) =F_2 \text{ in }\Omega.
 	\]
 	This proves the assertion.
 \end{proof}

 \begin{rmk}
 	It might be interesting to consider the nonlocal variable elliptic operator $\mathcal{L}^s=(-\nabla \cdot (\sigma \nabla))^s$ to replace the fractional Laplacian $(-\Delta)^s$ in Theorem \ref{Main Thm 1} and \ref{Main Thm 2}, when $\mathbf{a}(x,u)$ is not linear with respect to $u$. However, due to the lack of suitable H\"older estimates for the linear nonlocal elliptic equation $\LC \mathcal{L}^s +a^{(1)}\RC u=F$ in $\Omega$, $u=0$ in $\Omega_e$, we do not know how to prove the local well-posedness for $\mathcal{L}^s u+\mathbf{a}(x,u)=0$ in $\Omega$. Hence, we are not able to use $\mathcal{L}^s$ to replace $(-\Delta)^s$ under the nonlinear setting.
 \end{rmk}

 \section{Unique determination with reduced unknowns}\label{Sec 5}

 In this section, let us  show a unique determination result for a special case of \eqref{main equation fractional}, when 
 \begin{align}\label{qudratic}
 	\mathbf{a}(x,u)=a^{(2)}u^2,
 \end{align} 
 where $a^{(2)}\in C^s(\overline{\Omega})$.  That is, we give the proof of Theorem \ref{Thm: qudratic}.
 When $\mathbf{a}(x,y)$ is of the form \eqref{qudratic}, it automatically satisfies \eqref{condition a polynomial}. We are interested to determine $a^{(2)}$ and $F$ for a fractional semilinear elliptic equation 
 \begin{align}\label{semi quadratic}
 	\begin{cases}
 		(-\Delta)^s u +a^{(2)}u^2=F &\text{ in }\Omega, \\
 		u=f&\text{ in }\Omega_e.
 	\end{cases}
 \end{align}
 By using preceding results, it is known that we can use three measurements to determine both $a^{(2)}$ and $F$ in \eqref{semi quadratic}. As a matter of fact, we are able to reduce the number of measurements. Before doing so, let us state an interesting result, which can be regarded as a comparison principle for  fractional semilinear elliptic equations. To our best knowledge, this result is new to the literature.

 \begin{prop}[Comparison principle]\label{Thm Comparison}
 	Let $\Omega \subset \R^n$ be a $C^{1,1}$ domain, for $n\geq 1$, and $s\in (0,1)$. Let $\mathbf{b}(x,u)\geq 0$ be a bounded function, for $(x,y)\in \Omega \times \R$. If $u\in H^s(\R^n)$ is a weak solution of 
 	\begin{align}\label{eq comparison}
 		\begin{cases}
 			(-\Delta)^s u +\mathbf{b}(x,u)=F &\text{ in }\Omega ,\\
 			u=f&\text{ in }\Omega_e.
 		\end{cases}
 	\end{align}
 	Suppose that $F\leq 0$ in $\Omega$ and $f\leq 0$ in $\Omega_e$ are bounded functions. Then $u\leq 0$ in $\Omega$.
 \end{prop}
 
 Note that the function $\mathbf{b}(x,u)$ may not need to satisfy the condition \eqref{condition a polynomial}, but the positivity of $\mathbf{b}(x,u)$ plays an essential role in the upcoming arguments.
 Recall that a solution $u$ is called a weak solution provided that it satisfies the weak formulation
 \begin{align}\label{weak formulation}
 	\int_{\R^n}(-\Delta)^{s/2} u \cdot (-\Delta)^{s/2}\varphi \, dx + \int_{\Omega} \mathbf{b}(x,u)\varphi \, dx =\int_{\Omega}F\varphi \, dx,
 \end{align}
 for any $\varphi \in \wt H^s(\Omega)$.
 
 \begin{proof}[Proof of Theorem \ref{Thm Comparison}]
 	Since $u$ is a weak solution of \eqref{eq comparison}, one can write it in terms of the weak formulation \eqref{weak formulation}, for any $\varphi \in \wt H^s(\Omega)$. Notice that the first term in \eqref{weak formulation} can be expressed as 
 	\begin{align*}
 		&\int_{\R^n}(-\Delta)^{s/2} u \cdot (-\Delta)^{s/2}\varphi \, dx \\
 		= & \iint_{\R^{2n}} \frac{(u(x)-u(y))(\varphi(x)-\varphi(y))}{|x-y|^{n+2s}} \, dydx\\
 		=&\iint_{\R^{2n}\setminus (\Omega_e\times \Omega_e)} \frac{(u(x)-u(y))(\varphi(x)-\varphi(y))}{|x-y|^{n+2s}} \, dydx,
 	\end{align*}
 	where we utilized $\varphi\equiv 0$ in $\Omega_e$.
 	By writing $u=u^+-u^- \in H^s(\R^n)$, where $u^+=\max \{u,0\}\in H^s(\R^n)$ and $u^-=\max\{-u,0\}\in H^s(\R^n)$. Taking the test function $\varphi=u^+$, one can see that $\varphi\equiv0$ in $\Omega_e$ directly. We further assume that $u^+$ is not identically zero, which will leads a contradiction via the following arguments.
 	
 	First, by inserting $\varphi=u^+$, we write 
 	\begin{align*}
 		\iint_{\R^{2n}\setminus (\Omega_e\times \Omega_e)} \frac{(u(x)-u(y))\LC u^+(x)-u^+(y)\RC }{|x-y|^{n+2s}} \, dydx=I+II,
 	\end{align*}
 	where 
 	\begin{align*}
 		I:=&\iint_{\Omega \times \Omega} \frac{(u(x)-u(y))\LC u^+(x)-u^+(y)\RC }{|x-y|^{n+2s}} \, dydx,\\
 		II:=&2\int_{\Omega}\int_{\Omega_e} \frac{(u(x)-u(y))\LC u^+(x)-u^+(y)\RC }{|x-y|^{n+2s}} \, dydx.
 	\end{align*}
 	For $I$, we observe that $\LC u^-(x)-u^-(y)\RC \LC u^+(x)-u^+(y)\RC\leq 0$, so 
 	\begin{align}\label{estimate I}
 		I \geq \iint_{\Omega \times \Omega} \frac{ \LC u^+(x)-u^+(y)\RC^2 }{|x-y|^{n+2s}} \, dydx > 0.
 	\end{align}
 	The last equality of \eqref{estimate I} cannot hold since $u^+$ cannot be a constant function. If $u^+$ is a constant, say $u^+=\alpha_0>0$ in $\Omega$, via the definition \eqref{fractional Laplacian} of the fractional Laplacian, one has that 
 	\begin{align*}
 		(-\Delta)^s u(x)=& c_{n,s} \mathrm{P.V.}\int_{\R^n}\frac{u(x)-u(y)}{|x-y|^{n+2s}}\, dy \\
 		=& c_{n,s}\int_{\Omega_e}\frac{\alpha_0-f(y)}{|x-y|^{n+2s}}\, dy>0,
 	\end{align*}
 	for $x\in \Omega$, where we used $u|_{\Omega_e}=f\leq 0$. Hence, via \eqref{eq comparison}, we have 
 	\[
 	0<(-\Delta)^s u +\mathbf{b}(x,u)=F\leq 0 \text{ in }\Omega,
 	\]
 	which contradicts to that $u^+=\alpha_0$ is a constant in $\Omega$. Thus, $u^+$ cannot be a constant in $\Omega$, so that \eqref{estimate I} must hold.
 	
 	Second, since $f\leq 0$ in $\Omega_e$, and 
 	$$
 	u(x)u^+(x)=\LC u^+(x)-u_-(x)\RC u^+(x)\geq 0 \text{ in } \Omega,
 	$$ 
 	then $II$ can be estimated by 
 	\begin{align*}
 		II= &2\int_{\Omega}\int_{\Omega_e} \frac{(u(x)-u(y))\LC u^+(x)-u^+(y)\RC }{|x-y|^{n+2s}} \, dydx \\
 		=& 2\int_{\Omega}\int_{\Omega_e} \frac{(u(x)-f(y))u^+(x)}{|x-y|^{n+2s}} \, dydx \\
 		\geq &0.
 	\end{align*}
 	Hence, 
 	\[
 	\iint_{\R^{2n}\setminus (\Omega_e\times \Omega_e)} \frac{(u(x)-u(y))\LC u^+(x)-u^+(y)\RC }{|x-y|^{n+2s}} \, dydx=I+II>0.
 	\]
 	Moreover, it is not hard to check that $\int_{\Omega} \mathbf{b}(x,u)u^+\, dx\geq 0$ since $\mathbf{b}(x,u)\geq 0$, and $\int_{\Omega} Fu^+ \, dx \leq 0$ since $F\leq 0$. With these relations at hand, However, by \eqref{weak formulation}, we then conclude that 
 	\begin{align*}
 		0<\int_{\R^n}(-\Delta)^{s/2} u \cdot (-\Delta)^{s/2}\varphi \, dx + \int_{\Omega} \mathbf{b}(x,u)\varphi \, dx =\int_{\Omega}F\varphi \, dx\leq 0,
 	\end{align*}
 	which leads to a contradiction. Therefore, $u^+$ must be identically zero, and this completes the proof.
 \end{proof}
 
 We are ready to prove Theorem \ref{Thm: qudratic}.

 \begin{proof}[Proof of Theorem \ref{Thm: qudratic}]
 	Similar as before, let us denote $u_j$ and $u_j^{(0)}$ to be the solutions of \eqref{equation in thm 3} and 	
 	\begin{align}\label{equation 0 in proof thm 4}
 		\begin{cases}
 			(-\Delta)^s u_j^{(0)}+a_j^{(2)} (u_j^{(0)})^2= F_j & \text{ in }\Omega,\\
 			u_j^{(0)}=0 &\text{ in } \Omega_e,
 		\end{cases}
 	\end{align}
 	for $j=1,2$. Via the condition \eqref{DN fractional 3}, it is not hard to see that 
 	\begin{align}\label{equal condition}
 		u:=u_1=u_2 \quad \text{ and }\quad u_0:=u_1^{(0)}=u_2^{(0)} \text{ in }\R^n.
 	\end{align}
 	Therefore, by subtracting \eqref{equation in thm 3} and \eqref{equation 0 in proof thm 4}, we have 
 	\begin{align}\label{equal condition 1}
 		\begin{cases}
 			(-\Delta)^s (u-u_0) + a_1^{(2)}\LC u^2-(u_0)^2 \RC =0 & \text{ in }\Omega,\\
 			(-\Delta)^s (u-u_0) +a_2^{(2)} \LC u^2 -(u_0)^2 \RC =0 & \text{ in }\Omega.
 		\end{cases}
 	\end{align}
 	Subtracting the above two identities yield that 
 	\begin{align}\label{some id in proof 4}
 		\LC a_1^{(2)}-a_2^{(2)} \RC \LC u+u_0 \RC \LC u-u_0\RC =0 \quad \text{ in }\Omega,
 	\end{align}
 	It remains to analyze \eqref{some id in proof 4}.
 	
 	We first claim that 
 	\begin{align}\label{nonzero 1}
 		u+u_0 \neq 0 \text{ in }\Omega.
 	\end{align}
 	Note that we have conditions that $q_j\geq0$, $F_j \leq 0$ in $\Omega$, for $j=1,2$ and $f\leq 0$ in $\Omega_e$. By considering $\mathbf{b}(x,u_j)$ in Theorem \ref{Thm Comparison} as $\mathbf{b}(x,u_j)=a_j^{(2)}u_j^2\geq 0$, and apply the conclusion of Theorem \ref{Thm Comparison}, then we have that the solutions $u$ and $u_0$ must be nonpositive. Moreover, $u$ and $u_0$ are not identically zero, if $a_j^{(2)}$, $F_j$ and $f$ are not identically zero. By $u,u_0\leq 0$, the claim \eqref{nonzero 1} holds.

 	We next claim 
 	\begin{align}\label{nonzero 2}
 		u-u_0\neq0 \text{ in }\Omega 
 	\end{align}
 	as well.
 	If $u-u_0=0$ in $\Omega$, the equation \eqref{equal condition 1} yields that $(-\Delta)^s (u-u_0)=0$ in $\Omega$, then the strong unique continuation (Proposition \ref{Prop:(Uniqueness-theorem)}) implies that $u\equiv u_0$ in $\R^n$. However, 
 	$$
 	u|_{\Omega_e}=f\neq 0=u_0|_{\Omega_e},
 	$$
 	which leads a contradiction. Thus, \eqref{nonzero 2} holds. 
 	
 	Finally, with \eqref{nonzero 1} and \eqref{nonzero 2} at hand, combining with \eqref{some id in proof 4}, we have $q_1-q_2=0$ or $q_1=q_2$ in $\Omega$. Therefore,
 	\[
 	F_1=(-\Delta)^s u +a_1^{(2)} u^2 =(-\Delta)^s u +a_2^{(2)}u^2=F_2 \text{ in }\Omega,
 	\]
 	which proves the theorem.
 \end{proof}

% \medskip 
% 
%\section*{Concluduing Remarks.}
% 
% Let us summarize what we have investigated in this work.
% 
% \begin{itemize}
% 	\item[(a)] We prove Theorems \ref{Main Thm 1}  by using a minimal number of measurements and infinitely many measurements, respectively. The proofs are mainly based on the strong unique continuation property for the fractional Laplacian Laplacian operator $(-\Delta)^s $. 
% 	
% 	\item[(b)] For Theorem \ref{Main Thm 2}, we need to utilize infinitely many measurements due to the technique of the higher order linearization. Meanwhile, since the nonlinearity $\mathbf{a}(x,u)=\sum_{k=1}^\infty a_k(x)u^k$ satisfies \eqref{condition a polynomial} as $N=\infty$, which has infinitely many Taylor's coefficient $a^{(k)}(x)$, for any $k\in \N$. Hence, it is natural to determine $\mathbf{a}(x,u)$ via infinitely many measurements.
% 	
% 	\item[(c)] We used two measurements to prove the simultaneously recovery for both source $F$ and coefficient $q$ in Theorem \ref{Thm: qudratic}, where the unknowns have been reduced. It seems that the optimal number of measurements equal to the number of unknown functions in fractional inverse problems. In \cite{cekic2020calderon}, the authors also used $(N+1)$ measurements to determine $(N+1)$ unknown coefficients uniquely for the linear fractional Schr\"odinger equation with drift.
% \end{itemize}
% 
% 

\bigskip

\noindent\textbf{Acknowledgment.} 
Y.-H. Lin is partially supported by the Ministry of Science and Technology Taiwan, under the project 111-2628-M-A49-002. H Liu is supported by the Hong Kong RGC General Research Fund (projects 12302919, 12301218 and 11300821) and the NSFC/RGC Joint Research Grant (project N\_CityU101/21).

\bibliographystyle{alpha}
\bibliography{ref}

\end{document}